\documentclass[12pt, reqno]{amsart}

\usepackage[a4paper, centering, total={170mm,240mm}]{geometry}

\usepackage{amssymb,amsmath,latexsym,amsthm,amsfonts}
\usepackage{blindtext}
\usepackage{todonotes}
\usepackage{esint,dsfont}

\usepackage[colorlinks=true, linkcolor=blue, citecolor=purple, urlcolor=blue]{hyperref}

\def\N{\mathbb{N}}

\def\Z{\mathbb{Z}}
\def\Q{\mathbb{Q}}

\def\F{\mathbb{F}}

\newtheorem{prop}{\bf Proposition}[section]
\newtheorem{thm}[prop]{\bf Theorem}

\newtheorem{lem}[prop]{\bf Lemma}

\begin{document}

\title[Hyperbolicity and uniformly Lipschitz affine actions on subspaces of $L^1$]{{\bf\Large Hyperbolicity and uniformly Lipschitz affine actions on subspaces of $L^1$}}

\author[I. Vergara]{Ignacio Vergara}
\address{Saint-Petersburg State University, 7/9 Universitetskaya Embankment, Saint Petersburg 199034, Russia.}
\curraddr{Departamento de Matem\'atica y Ciencia de la Computaci\'on, Universidad de Santiago de Chile, Las Sophoras 173, Estaci\'on Central, Chile.}

\email{ign.vergara.s@gmail.com}
\thanks{The research was carried out at the Leonhard Euler St. Petersburg International Mathematical Institute and supported by the Ministry of Science and Higher Education of the Russian Federation (Agreement № 075–15–2022–287 dated 06.04.2022).}

\makeatletter
\@namedef{subjclassname@2020}{%
  \textup{2020} Mathematics Subject Classification}
\makeatother

%


\begin{abstract}
We show that every hyperbolic group has a proper uniformly Lipschitz affine action on a subspace of an $L^1$ space. We also prove that every acylindrically hyperbolic group has a uniformly Lipschitz affine action on such a space with unbounded orbits. Our main tools are the $\mathbb{Q}$-bicombings on hyperbolic groups constructed by Mineyev and the characterisation of acylindrical hyperbolicity in terms of actions on quasi-trees by Balasubramanya.
\end{abstract}


\begingroup
\def\uppercasenonmath#1{} 
\let\MakeUppercase\relax 
\maketitle
\endgroup

\section{{\bf Introduction}}

This paper is a follow-up to \cite{Ver}, where we constructed proper affine uniformly Lipschitz actions on subspaces of $L^1$ for certain classes of groups, including residually finite hyperbolic groups. One of the goals of the present work is to remove the residual finiteness hypothesis from this result.

Let $\Gamma$ be a countable group and let $E$ be a Banach space. We denote by $\mathcal{B}(E)$ the algebra of bounded operators on $E$. An affine action of $\Gamma$ on $E$ is given by a representation $\pi:\Gamma\to\mathcal{B}(E)$ and a map $b:\Gamma\to E$ satisfying the cocycle identity:
\begin{align*}
b(st)=\pi(s)b(t)+b(s),\quad\forall s,t\in\Gamma.
\end{align*}
The affine action is then given by
\begin{align*}
s\cdot v=\pi(s)v+b(s),\quad\forall s\in\Gamma,\ \forall v\in E.
\end{align*}
Saying that this action is uniformly Lipschitz is equivalent to the fact that $\pi$ is uniformly bounded:
\begin{align*}
\sup_{s\in\Gamma}\|\pi(s)\| < \infty.
\end{align*}
In this case, the action is proper (resp. has bounded orbits) if and only if $b$ is proper (resp. bounded).

Here we are interested in subspaces of $L^1$. The relevance of these spaces in this context comes from the fact that the Haagerup property can be characterised in terms of proper affine isometric actions on $E\subset L^1$; see \cite[Corollary 1.5]{ChDrHa} and \cite[Corollary 6.23]{ChDrHa}. Thus, having a proper uniformly Lipschitz action on such a space can be regarded as a weaker form of the Haagerup property. We prove the following general criterion, which was already implicit in \cite{Ver}.

\begin{thm}\label{Thm_suff_cond}
Let $\Gamma$ be a countable group with identity element $e$, and let $H$ be a Hilbert space. Assume that there is a map $f:\Gamma\to H$ and a constant $M\geq 0$ such that
\begin{align*}
\|f(sx)-f(sy)\|^2\leq \|f(x)-f(y)\|^2 + M,\quad\forall s,x,y\in\Gamma.
\end{align*}
Then there exist a measure space $(\Omega,\mu)$, a closed subspace $E\subset L^1(\Omega,\mu)$, and a uniformly bounded representation $\pi:\Gamma\to\mathcal{B}(E)$ such that $\pi$ admits a cocycle $b:\Gamma\to E$ satisfying
\begin{align*}
\|b(s)\|_E=\|f(s)-f(e)\|_H +2,\quad\forall s\in\Gamma\setminus\{e\}.
\end{align*}
\end{thm}

This result (and its proof) is the general procedure used in \cite{Ver} to construct proper actions. For groups acting on products of quasi trees, $f$ is given by the semimetric $d_a$ defined in \cite[\S 3.3]{Ver}. More precisely,
\begin{align*}
\|f(s)-f(t)\|^2=d_a(s\cdot a, t\cdot a),\quad\forall s,t\in\Gamma.
\end{align*}
For weakly amenable groups with Cowling--Haagerup constant $1$, $f$ is the map $R:\Gamma\to H$ given by \cite[Proposition 5.1]{Ver}.

It turns out that we can define such a map for every hyperbolic group, without using actions on products of quasi-trees, which are known to exist only in the residually finite case; see \cite{BeBrFu}. It is worth mentioning, however, that the existence of a non-residually finite hyperbolic group is a major open problem in geometric group theory.

Recall that a hyperbolic group is a finitely generated group whose Cayley graph with respect to any (equivalently one) finite generating set is hyperbolic. We refer the reader to \cite[\S 7]{Loh} for a detailed treatment. For every hyperbolic group $\Gamma$, Mineyev \cite{Min} constructed a quasi-geodesic, $\Gamma$-equivariant $\Q$-bicombing with bounded area. This construction was used by Yu \cite{Yu} to prove that every hyperbolic group has a proper isometric affine action on $\ell^p$ for $p$ large enough. Here we use Mineyev's result to prove the following.

\begin{thm}\label{Thm_hyp}
Every hyperbolic group admits a proper affine uniformly Lipschitz action on a subspace of an $L^1$ space.
\end{thm}

In the residually finite case, Dru\c{t}u and Mackay proved that there is such an action on the whole space $\ell^1$; see \cite{Dru_slides}.

Finally, we turn our attention to acylindrically hyperbolic groups. A group is said to be acylindrically hyperbolic if it admits a non-elementary acylindrical action on a hyperbolic space. We refer the reader to \cite{Osi} and \cite{Osi2} for details. This class includes non-elementary hyperbolic groups, (non-virtually abelian) mapping class groups, $\operatorname{Out}(\F_N)$ ($N\geq 2$), and infinitely presented graphical $Gr(7)$ small cancellation groups. In \cite{Bal}, Balasubramanya characterised acylindrical hyperbolicity by the existence of a non-elementary acylindrical action on a quasi-tree. This is our main tool for establishing the following result.

\begin{thm}\label{Thm_acyl}
Every acylindrically hyperbolic group admits an affine uniformly Lipschitz action on a subspace of an $L^1$ space with unbounded orbits.
\end{thm}

Let us point out that this result was obtained independently by Dru\c{t}u and Mackay; see \cite{Dru_slides}. In fact, their result is stronger, as they construct actions on $\ell^1$ instead of a subspace. Nevertheless, we believe Theorem \ref{Thm_acyl} is relevant because the techniques used in the proof are different.

Finally, observe that Theorem \ref{Thm_acyl} cannot be improved in order to obtain proper actions for all acylindrically hyperbolic groups. Indeed, such an action yields a coarse embedding into a Hilbert space; see \cite[Corollary 6.2]{Ver}. On the other hand, if $\Gamma$ does not admit a coarse embedding, neither does $\Gamma\ast\Z$, and this group is acylindrically hyperbolic; see e.g. \cite[Corollary 2.2]{MinOsi}. For more examples of acylindrically hyperbolic groups without coarse embeddings, see \cite{GruSis}.

\subsection*{Organisation of the paper}
In Section \ref{S_general}, we prove Theorem \ref{Thm_suff_cond}. Section \ref{S_bic} is devoted to $\Q$-bicombings on hyperbolic groups and the proof of Theorem \ref{Thm_hyp}. Finally, in Section \ref{S_acyl}, we look at actions on quasi-trees and prove Theorem \ref{Thm_acyl}.

\section{{\bf A sufficient condition for actions on $E\subset L^1$}}\label{S_general}

We begin by proving Theorem \ref{Thm_suff_cond}. The proof is essentially the same as that of \cite[Theorem 1.1]{Ver}, but in a more general context. We refer the reader to \cite[\S 4]{Ver} for more details on some of our arguments.

\begin{proof}[Proof of Theorem \ref{Thm_suff_cond}]
Let $f:\Gamma\to H$ be the map given by the hypothesis, and let $V$ be the vector space of real-valued, finitely supported functions on $\Gamma$ with mean $0$. For $v\in V$, we define
\begin{align*}
\|v\|_f = \left(-\frac{1}{2}\sum_{x,y\in\Gamma}v(x)v(y)\|f(x)-f(y)\|_H^2\right)^\frac{1}{2},
\end{align*}
and we let $E$ be the completion of $V$ for the norm
\begin{align*}
\|v\|_E = \|v\|_f + \|v\|_1,
\end{align*}
where
\begin{align*}
\|v\|_1 = \sum_{x\in\Gamma}|v(x)|.
\end{align*}
Then there exists a measure space $(\Omega,\mu)$ such that $E$ embeds into $L^1(\Omega,\mu)$ isometrically; see \cite[Lemma 4.1]{Ver}. For all $v\in V$ and $s\in\Gamma$, let us define $\pi(s)v$ by
\begin{align*}
\pi(s)v(x)=v(s^{-1}x),\quad\forall x\in\Gamma.
\end{align*} 
Now observe that
\begin{align*}
\|\pi(s)v\|_f^2-\|v\|_f^2 &= \frac{1}{2}\sum_{x,y\in\Gamma}v(x)v(y)\left(\|f(x)-f(y)\|_H^2-\|f(sx)-f(sy)\|_H^2\right)\\
&\leq \frac{M}{2}\left(\sum_{x\in\Gamma}|v(x)|\right)^2.
\end{align*}
This implies that
\begin{align*}
\|\pi(s)v\|_E &= \|\pi(s)v\|_f + \|\pi(s)v\|_1\\
&\leq \sqrt{\frac{M}{2}}\|v\|_1 + \|v\|_f + \|v\|_1\\
&\leq \left(\sqrt{\frac{M}{2}}+1\right)\|v\|_E.
\end{align*}
Therefore $\pi$ extends to a uniformly bounded representation on $E$. Finally, we define $b:\Gamma\to E$ by
\begin{align*}
b(s)=\delta_s-\delta_e,\quad\forall s\in\Gamma.
\end{align*}
Then $b$ is a cocycle for $\pi$, and
\begin{align*}
\|b(s)\|_E &= \|b(s)\|_f + \|b(s)\|_1\\
&= \|f(s)-f(e)\|_H + 2,
\end{align*}
for all $s\in\Gamma\setminus\{e\}$.
\end{proof}

\section{{\bf From $\Q$-bicombings to affine actions on $E\subset L^1$}}\label{S_bic}

In this section we prove Theorem \ref{Thm_hyp}. For this purpose, we recall the definition of $\Q$-bicombings, in the sense of \cite{Min}.

Let $\Gamma$ be a finitely generated group and let $X=(X^{(0)},X^{(1)})$ be its Cayley graph with respect to a finite generating set. Here we denote by $X^{(0)}$ and $X^{(1)}$ the set of vertices and (oriented) edges respectively. We view $X^{(0)}$ as a metric space by endowing it with the edge-path distance. In other words, for $x,y\in X^{(0)}$, $d(x,y)$ is defined as the length of the shortest path (without orientation) between $x$ and $y$.

A bicombing on $X$ is a map $p$ associating to each ordered pair of vertices $(x,y)$ an oriented edge-path $p[x,y]$ from $x$ to $y$. We say that $p$ is geodesic if $p[x,y]$ is a geodesic path for every $x,y\in X^{(0)}$.

We denote by $\Q[X^{(1)}]$ and $\Q[X^{(0)}]$ the space of $1$-chains and $0$-chains respectively. These are finitely supported functions on $X^{(i)}$ ($i=0,1$) with values in $\Q$. The boundary map $\partial:\Q[X^{(1)}]\to\Q[X^{(0)}]$ is defined on each edge $(u,v)\in X^{(1)}$ by
\begin{align*}
\partial(u,v)=v-u,
\end{align*}
and extended by linearity. Observe that, if $p$ is a bicombing, we may view $p[x,y]$ as a $1$-chain such that $\partial p[x,y]=y-x$.

The set $X^{(1)}$ forms a basis for $\Q[X^{(1)}]$, and we can thus consider the $\ell^1$ norm with respect to this basis. More precisely, for all $\alpha\in\Q[X^{(1)}]$, we define
\begin{align*}
\|\alpha\|_1=\sum_{\omega\in X^{(1)}}|\alpha(\omega)|.
\end{align*}
Observe that we can view $\alpha$ as a function on $X^{(0)}\times X^{(0)}$ such that $\alpha(u,v)=0$ if $(u,v)\notin X^{(1)}$. This allows us to write
\begin{align}\label{l1_norm_alpha}
\|\alpha\|_1=\sum_{u,v\in X^{(0)}}|\alpha(u,v)|.
\end{align}

A $\Q$-bicombing is a map $q$ that assigns, to each ordered pair of vertices $(x,y)$, a $1$-chain $q[x,y]$ such that $\partial q[x,y]=y-x$. We say that $q$ is quasi-geodesic if
\begin{itemize}
\item[$\bullet$] there exists a geodesic bicombing $p$ and a constant $r>0$ such that, for every $x,y\in X^{(0)}$, the support of $q[x,y]$ is contained in an $r$-neighbourhood of $p[x,y]$,
\item[$\bullet$] there is a constant $C>0$ such that, for all $x,y\in X^{(0)}$,
\begin{align*}
\|q[x,y]\|_1\leq Cd(x,y).
\end{align*}
\end{itemize}

The following result will be crucial for our purposes; see \cite[Theorem 10]{Min}.

\begin{thm}[Mineyev]\label{Thm_Min}
Let $\Gamma$ be a hyperbolic group and let $X$ be its Cayley graph with respect to a finite generating set. There exists a quasi-geodesic $\Q$-bicombing $q$ on $X$ with the following properties:
\begin{itemize}
\item[a)] $q$ is $\Gamma$-equivariant: $sq[x,y]=q[sx,sy]$, $\forall s\in\Gamma$, $\forall x,y\in X^{(0)}$.
\item[b)] $q$ is anti-symmetric: $q[y,x]=-q[x,y]$, $\forall x,y\in X^{(0)}$.
\item[c)] $q$ has bounded area: there exists $M>0$ such that, for all $x,y,z\in X^{(0)}$,
\begin{align*}
\left\|q[x,y]+q[y,z]+q[z,x]\right\|_1\leq M.
\end{align*}
\end{itemize}
\end{thm}

We will use these $\Q$-bicombings, together with Theorem \ref{Thm_suff_cond} in order to prove Theorem \ref{Thm_hyp}. For this purpose, we need the following lemma, which gives a lower bound for the $\ell^1$ norm of $1$-chains.

\begin{lem}\label{Lem_norm_alpha}
Let $(X^{(0)},X^{(1)})$ be a finite directed graph. Let $\alpha\in\Q[X^{(1)}]$ be a $1$-chain, and assume that there exist $x,y\in X^{(0)}$ such that
\begin{align*}
\partial\alpha=y-x.
\end{align*}
Then
\begin{align*}
\|\alpha\|_1\geq d(x,y).
\end{align*}
\end{lem}
\begin{proof}
We will proceed by induction on $n=d(x,y)$. If $n=1$, by \eqref{l1_norm_alpha}, we have
\begin{align*}
\|\alpha\|_1&\geq\frac{1}{2}\sum_{v\in X^{(0)}}\left(|\alpha(x,v)|+|\alpha(v,x)|+|\alpha(y,v)|+|\alpha(v,y)|\right)\\
&\geq \frac{1}{2}\left|\sum_{v\in X^{(0)}}\alpha(v,x)-\alpha(x,v)\right|+\frac{1}{2}\left|\sum_{v\in X^{(0)}}\alpha(v,y)-\alpha(y,v)\right|\\
&=\frac{1}{2}|\partial\alpha(x)|+\frac{1}{2}|\partial\alpha(y)|\\
&=1.
\end{align*}
Now assume that the result holds for $n$, and that $d(x,y)=n+1$. We will define a new graph $(Y^{(0)},Y^{(1)})$ by collapsing all the neighbours of $x$ to a single point. More precisely, the set of vertices will be
\begin{align*}
Y^{(0)}=\left(X^{(0)}\setminus\{u_1,\ldots,u_m\}\right)\cup\{x'\},
\end{align*}
where $u_1,\ldots,u_m$ are the neighbours of $x$ in $(X^{(0)},X^{(1)})$, and $x'$ is a new vertex. The set of edges will be given by
\begin{align*}
Y^{(1)}=&\{(x,x'),(x',x)\}\cup\left\{(x',v)\,\mid\, v\in Y^{(0)}\setminus\{x,x'\},\ \exists i\in\{1,\ldots,m\}, (u_i,v)\in X^{(1)}\right\}\\
&\cup \left\{(v,x')\,\mid\, v\in Y^{(0)}\setminus\{x,x'\},\ \exists i\in\{1,\ldots,m\}, (v,u_i)\in X^{(1)}\right\}\\
&\cup \left\{(u,v)\in X^{(1)}\,\mid\, u,v\notin\{x,u_1,\ldots,u_m\}\right\}.
\end{align*}
Let us denote by $d_Y$ the edge-path distance on this new graph. Then $d_Y(x',y)=n$. We can define a $1$-chain $\beta\in\Q[Y^{(1)}]$ as follows. 
\begin{align*}
\beta(x,x')=\beta(x',x)=0.
\end{align*}
For all $v\in Y^{(0)}\setminus\{x\}$,
\begin{align*}
\beta(x',v)&=\sum_{i=1}^m\alpha(u_i,v),\\
\beta(v,x')&=\sum_{i=1}^m\alpha(v,u_i),
\end{align*}
where, again, we view $\beta$ as a function on $Y^{(0)}\times Y^{(0)}$ with value $0$ where there is no edge. For $u,v\in Y^{(0)}\setminus\{x,x'\}$, we simply put
\begin{align*}
\beta(u,v)=\alpha(u,v).
\end{align*}
We claim that $\partial\beta=y-x'$. Indeed,
\begin{align*}
\partial\beta(x)&=\beta(x',x)-\beta(x,x')\\
&=0,
\end{align*}
and
\begin{align*}
\partial\beta(x')&=\beta(x,x')-\beta(x',x)+\sum_{v\in Y^{(0)}\setminus\{x\}}\beta(v,x')-\beta(x',v)\\
&=\sum_{v\in Y^{(0)}\setminus\{x\}}\sum_{i=1}^m\alpha(v,u_i)-\alpha(u_i,v)\\
&=\sum_{i=1}^m\partial\alpha(u_i)-\alpha(x,u_i)+\alpha(u_i,x)\\
&=\sum_{i=1}^m\alpha(u_i,x)-\alpha(x,u_i)\\
&=\partial\alpha(x)\\
&=-1.
\end{align*}
Now, for $w\in Y^{(0)}\setminus\{x,x'\}$,
\begin{align*}
\partial\beta(w)&=\beta(x',w)-\beta(w,x')+\sum_{v\in Y^{(0)}\setminus\{x,x'\}}\beta(v,w)-\beta(w,v)\\
&=\sum_{i=1}^m\alpha(u_i,w)+\alpha(w,u_i)+\sum_{v\in Y^{(0)}\setminus\{x,x'\}}\alpha(v,w)-\alpha(w,v)\\
&=\partial\alpha(w).
\end{align*}
This shows that $\partial\beta=y-x'$. Hence, by the induction hypothesis,
\begin{align*}
\|\beta\|_1\geq n.
\end{align*}
On the other hand,
\begin{align*}
\|\beta\|_1&=\left(\sum_{v\in Y^{(0)}\setminus\{x\}}|\beta(x',v)|+|\beta(v,x')|\right)+\sum_{u,v\in Y^{(0)}\setminus\{x,x'\}}|\beta(u,v)|\\
&\leq\left(\sum_{v\in Y^{(0)}\setminus\{x\}}\sum_{i=1}^m|\alpha(u_i,v)|+|\alpha(v,u_i)|\right)+\sum_{u,v\in Y^{(0)}\setminus\{x,x'\}}|\alpha(u,v)|\\
&=\|\alpha\|_1-\left(\sum_{i=1}^m|\alpha(u_i,x)|+|\alpha(x,u_i)|\right)\\
&\leq\|\alpha\|_1-\left|\sum_{i=1}^m\alpha(u_i,x)-\alpha(x,u_i)\right|\\
&=\|\alpha\|_1-|\partial\alpha(x)|\\
&=\|\alpha\|_1-1.
\end{align*}
Therefore
\begin{align*}
\|\alpha\|_1\geq\|\beta\|_1+1\geq n+1=d(x,y).
\end{align*}
\end{proof}

Now we are ready to prove Theorem \ref{Thm_hyp}.

\begin{proof}[Proof of Theorem \ref{Thm_hyp}]
Let $\Gamma$ be a hyperbolic group and let $X$ be its Cayley graph with respect to a finite generating set. Let $q$ be the $\Q$-bicombing given by Theorem \ref{Thm_Min}. By \cite[Theorem 6.8]{ChDrHa}, there is a Hilbert space $H$ and a map $J:\ell^1(X^{(1)})\to H$ such that
\begin{align*}
\|J(v)-J(w)\|_H^2=\|v-w\|_1,\quad\forall v,w\in \ell^1(X^{(1)}).
\end{align*}
Now define $f:\Gamma\to H$ by
\begin{align*}
f(x)=J(q[e,x]),\quad\forall x\in\Gamma.
\end{align*}
Thus, for every $s,x,y\in\Gamma$,
\begin{align*}
\|f(sx)-f(sy)\|_H^2 &= \|q[e,sx]-q[e,sy]\|_1\\
&\leq \|q[e,sx]+q[sx,sy]+q[sy,e]\|_1 + \|q[sy,sx]+q[sx,s]+q[s,sy]\|_1\\
&\qquad + \|q[s,sx]+q[sy,s]\|_1\\
&\leq 2M + \|q[e,x]-q[e,y]\|_1\\
&= 2M + \|f(x)-f(y)\|_H^2.
\end{align*}
In the last inequality we used the fact that $q$ is $\Gamma$-equivariant and that the $\ell^1$ norm is $\Gamma$-invariant. By Theorem \ref{Thm_suff_cond}, $\Gamma$ has a uniformly bounded representation on $E\subset L^1$ admitting a cocycle $b:\Gamma\to E$ such that
\begin{align*}
\|b(s)\|_E &=\|f(s)-f(e)\|_H +2\\
&= \|q[e,s]\|_1^{1/2}+2,
\end{align*}
for all $s\in\Gamma\setminus\{e\}$. Now fix $s\in\Gamma$ and let $Y$ be the finite subgraph of $X$ given by the support of $q[e,s]$. If $d_X$ and $d_Y$ denote the edge-path distances on $X$ and $Y$ respectively, we have
\begin{align*}
d_Y(e,s)\geq d_X(e,s).
\end{align*}
On the other hand, by Lemma \ref{Lem_norm_alpha},
\begin{align*}
\|q[e,s]\|_1 \geq d_Y(e,s)\geq d_X(e,s).
\end{align*}
Since $s$ was arbitrary, we conclude that
\begin{align*}
\|q[e,s]\|_1 \geq d_X(e,s)
\end{align*}
for all $s\in\Gamma$. This shows that $b$ is proper.
\end{proof}

\section{{\bf From quasi-trees to affine actions on $E\subset L^1$}}\label{S_acyl}

Now we turn to acylindrically hyperbolic groups. Our main tool is the following result of Balasubramanya; see \cite[Theorem 1.7]{Bal}. Recall that a quasi-tree is a graph which is quasi-isometric to a tree when we endow it with the edge-path distance.

\begin{thm}[Balasubramanya]\label{Thm_Bal}
Every acylindrically hyperbolic group admits a non-elementary cobounded acylindrical action on a quasi-tree.
\end{thm}

This theorem provides a characterisation of acylindrical hyperbolicity by the existence of a non-elementary acylindrical action on a quasi-tree. In particular, such actions have unbounded orbits. With this we can prove Theorem \ref{Thm_acyl}.

\begin{proof}[of Theorem \ref{Thm_acyl}]
By Theorem \ref{Thm_Bal}, $\Gamma$ has an isometric action on a quasi-tree $X$ with unbounded orbits. By \cite[\S 3]{Ver}, there is a conditionally negative definite kernel $d_a:X\times X\to\N$ and a constant $\Delta_X\geq 0$ such that
\begin{align*}
d(x,y)-\Delta_X \leq d_a(x,y) \leq d(x,y),\quad\forall x,y\in X,
\end{align*}
where $d$ is the edge-path distance on $X$. The fact that $d_a$ is conditionally negative definite means that there is a Hilbert space $H$ and a map $f:X\to H$ such that
\begin{align*}
d_a(x,y)=\|f(x)-f(y)\|_H^2,\quad\forall x,y\in X.
\end{align*}
In particular,
\begin{align*}
\|f(s\cdot x)-f(s\cdot y)\|_H^2 \leq \|f(x)-f(y)\|_H^2+\Delta_X, \quad\forall s\in\Gamma,\ \forall x,y\in X.
\end{align*}
Fix $x_0\in X$ and restrict $f$ to the orbit $\Gamma\cdot x_0$. This defines a map from $\Gamma$ to $H$ by
\begin{align*}
s\mapsto f(s\cdot x_0).
\end{align*}
By Theorem \ref{Thm_suff_cond}, $\Gamma$ has a uniformly bounded representation on $E\subset L^1$ admitting a cocycle $b:\Gamma\to E$ such that
\begin{align*}
\|b(s)\|_E^2 &\geq \|f(s\cdot x_0)-f(x_0)\|_H^2\\
&\geq d(s\cdot x_0,x_0)-\Delta_X,
\end{align*}
for all $s\in\Gamma$. Since the action $\Gamma\curvearrowright X$ has unbounded orbits, so does $\Gamma\curvearrowright E$.
\end{proof}

\subsection*{Acknowledgements}
I am grateful to Thomas Delzant for his very valuable comments and suggestions, including the use of bicombings to prove Theorem \ref{Thm_hyp}. I also thank the anonymous referee for pointing out a conceptual mistake in a previous version of this paper, which led to the incorporation of Lemma \ref{Lem_norm_alpha}.

\bibliographystyle{plain} 

\bibliography{Bibliography}

\end{document}